\documentclass[BCOR=12mm, DIV=default,twoside, pagesize, draft]{scrartcl}

\addtokomafont{caption}{\small}
\setkomafont{captionlabel}{\sffamily\bfseries}

\usepackage[final]{graphicx}
\usepackage{tikz}    
\usetikzlibrary{arrows,automata}   
\usepackage{todonotes}
\usepackage{amsmath, amssymb, amsthm} 
\usepackage[utf8x]{inputenc}
\usepackage{libertine} 
\usepackage[T1]{fontenc}
\usepackage[english]{babel}
\usepackage{url}  % f o r XeTex t o b r e a k l o n g URLs a t l i n e e n d i n g
\usepackage{faktor,csquotes,color}
\usepackage{pgfplots} 
\usepackage{enumitem}
\usepackage{stackengine}
\usepackage{breqn}

\usepackage{float} 
\usetikzlibrary{arrows}
\usetikzlibrary{datavisualization.formats.functions}
\pgfplotsset{compat=1.12}
\usepgfplotslibrary{fillbetween}
\usetikzlibrary{patterns}
\usepackage{accents}

\usepackage{dsfont}

\usepackage{changes}

\definechangesauthor[name={Kristoffer}, color=blue]{kri}
\newcommand{\stkout}[1]{\ifmmode\text{\sout{\ensuremath{#1}}}\else\sout{#1}\fi}
\setdeletedmarkup{\stkout{#1}}

\usepackage{units,varioref}

\usepackage[shortcuts]{extdash}

\usepackage{scrlayer-scrpage}
\KOMAoptions{draft=false}

\usepackage{mathtools}
\mathtoolsset{showonlyrefs,showmanualtags}
\numberwithin{equation}{section}

\allowdisplaybreaks 
\clearscrheadfoot
\automark[section]{section}
\ohead{\headmark}
\ofoot[\pagemark]{\pagemark}

\newtheoremstyle{break}{\topsep}{\topsep}{\itshape}{}{\bfseries}{.}{\newline}{}
\newtheoremstyle{exampl}{\topsep}{\topsep}{\upshape}{}{\bfseries}{.}{\newline}{}

\theoremstyle{plain}% default
\newtheorem{theorem}{Theorem}[section]
\newtheorem{lem}[theorem]{Lemma}  
\newtheorem{proposition}[theorem]{Proposition}  
 
\newtheorem{ass}[theorem]{Assumption}

\theoremstyle{definition}
\newtheorem{definition}[theorem]{Definition}
\newtheorem{example}[theorem]{Example}

\theoremstyle{remark}
\newtheorem{rem}[theorem]{Remark}

\def\F{\mathcal F}

\def\A{\mathcal A}

\def\E{\mathbb E}
\def\R{\mathbb{R}}

\def\I{\mathcal I}

\title{Dynkin ghost games with asymmetry and consolation}
\author{Erik Ekstr\"om\footnote{Department of Mathematics, Uppsala University, Sweden. 
		E-mail address: ekstrom@math.uu.se.} \and Yuqiong Wang\footnote{Department of Mathematics, Uppsala University, Sweden. 
		E-mail address: yuqiong.wang@math.uu.se.}}

\date{\today}

\begin{document}
	\maketitle

	\begin{abstract} 
We study a stopping game of preemption type between two players who both act under 
uncertain competition. In this framework we introduce, and study the effect of, (i) {\em asymmetry} of payoffs, allowing e.g. for different investments costs, and
(ii) {\em consolation}, i.e. partial compensation to the forestalled stopper. In general, 
this setting does not offer an explicit equilibrium. Instead, 
we provide a general verification theorem, which we then use  to explore various situations in which a solution can be constructed so that an equilibrium is obtained.
	\end{abstract}

	\vspace{1mm}
	

\section{Introduction} \label{sec:intro}

In many problems of strategic interaction in random environments, {\em incomplete} and {\em asymmetric} information are natural ingredients. For example, 
different agents may have access to different information regarding the profitability of a certain investment possibility,
and bidders in an auction may lack information about competitors'  private values. 
A particular case of incomplete/asymmetric information is that of \textbf{uncertain competition}. 
Such a feature is natural in, for example, many investment problems, for participants in online auctions where bidders are not aware of how many competitors they have, and in models of fraud detection.

In the current article we study optimal stopping games (Dynkin games) with uncertain competition. 
Such games were introduced in \cite{DE}, where it was assumed that two players can observe an underlying Markov process, and the one who stops first 
receives a payoff defined in terms of the underlying process. However, both players act under uncertain competition, and 
the urge to wait until the optimal (single-player) stopping time needs to be balanced against the risk of being forestalled by the other player (should they exist). 
In the main result of \cite{DE}, a Nash equilibrium is obtained in randomized stopping times using a guess-and-verify approach. 

Here we extend the framework of \cite{DE} to allow for 
(i) \textbf{asymmetric payoffs}, and (ii) \textbf{consolation} to the late stopper.
{\em Asymmetric payoffs} appear naturally in problems of investment timing, where different agents are subject to different investment costs. For example, 
the investment in a real option may incur larger hiring and operating costs for a small firm
than for a large one. Similarly, in problems of fraud detection, the different 
agents clearly have different objectives, leading naturally to asymmetric problems. The notion {\em consolation} refers to a situation in which 
a player receives a non-zero reward also when being forestalled.
This naturally arises in problems of investment timing with revenues depending on whether a monopoly or duopoly situation is obtained, so that the payoff of a particular player depends on the order of investment. We will assume that the consolation payoff is dominated by the payoff for the first stopper so that the game is still of preemption type.

\subsection{Literature review}

Stochastic dynamic games have been studied thoroughly by several authors, for example in connection with real option valuation and investment games in a duopoly setting (see e.g. \cite{DM}, \cite{DR}, \cite{Thi} and \cite{HW}). 
More recently, stochastic dynamic games with incomplete and asymmetric information have started to attract interest in the literature, see \cite{DGV}, \cite{G}, \cite{LP} and \cite{LM}.

A particular type of incomplete information is when 
there is uncertainty about the existence of competition.
For instance, the case when 
the existence of one player is known to the other player, but not vice versa, is relevant in studies of fraud detection; for a few studies along these lines, see \cite{BM}, \cite{DE}, \cite{EL}, \cite{ELO} and \cite{EMO}.
In particular, in \cite{DE} a symmetric optimal stopping game with uncertain competition and no consolation was studied. As mentioned above, the set-up of the current article extends the set-up in \cite{DE} to include for asymmetric payoff functions and consolation.

\subsection{Preview}
In Section~\ref{sec2} we provide a detailed formulation of the game under consideration, together with a few preliminary results.
Introducing asymmetric payoffs and consolation, the Dynkin ghost game becomes less tractable, and the full specification of an equilibrium is in general not 
possible.
In the absence of an explicit equilibrium we provide a verification result (Theorem~\ref{verif}), 
which specifies sufficient conditions for a candidate equilibrium to be an equilibrium.
In particular, the verification result requires the construction of two equilibrium value functions fulfilling certain martingale conditions. 
These martingale conditions can be translated into boundary conditions along stopping boundaries, and are thus instrumental when constructing candidate 
equilibrium values. In fact, some special cases are amenable to further analysis; Section~\ref{sec4} studies symmetric payoffs with consolation, and 
Section~\ref{sec7} explores asymmetric cases.

\section{Problem formulation}
\label{sec2}

To describe the game set-up, let $X$ be a continuous strong Markov process with state space $\mathcal I$, where $\mathcal I\subseteq\R$ is an interval, defined on a complete probability space $(\Omega, \mathcal{F},\mathbb P_x)$, where $\mathbb P_x(X_0=x)=1$.
We assume that $g_i, h_i:\mathcal I\to [0,\infty)$, $i=1,2$, are given  functions with $g_i\geq h_i$ and such that 
	\[\sup_{x\in \R} g_i(x)>0\]
(to rule out degenerate cases).
We denote by $\mathbb F:=(\mathcal F_t)_{t\geq 0}$ the augmentation of the smallest right-continuous filtration to which $X$ is adapted, and by $\mathcal T$ the set of $\mathbb F-$stopping times.
Given a constant discount rate $r\geq0$, let 
\[V^{g_i}(x):=\sup_{\tau\in\mathcal T}\E_x\left[e^{-r\tau}g_i(X_\tau)\right]\]
and 
\[V^{h_i}(x):=\sup_{\tau\in\mathcal T}\E_x\left[e^{-r\tau}h_i(X_\tau)\right]\]
be the value functions in the corresponding single-player games. 
Here (and in all similar expressions below) we use the convention that
$f(X_\tau)\mathds 1_{\{\tau=\infty\}}=0$ for a given function $f$. Also, let 
\[\tau_{g_i}:=\inf\{t\geq 0:V^{g_i}(X_t)\leq g_i(X_t)\}\]
and
\[\tau_{h_i}:=\inf\{t\geq 0:V^{h_i}(X_t)\leq h_i(X_t)\}\]
$i=1,2$.

\begin{ass}\label{ass}
We assume that the functions $g_i$, $V^{g_i}$ and $V^{h_i}$ are continuous. We also assume that the processes 
\[Y^{g_i}_t:=e^{-rt}V^{g_i}(X_t)\]
and
\[Y^{h_i}_t:=e^{-rt}V^{h_i}(X_t)\]
are supermartingales, and that the stopped versions 
$Y^{g_i}_{t\wedge\tau_{g_i}}$ and 
$Y^{h_i}_{t\wedge\tau_{h_i}}$ are uniformly integrable martingales.
\end{ass}

\begin{rem}
A consequence of the assumptions above is that $\tau_{g_i}$ is an
optimal stopping time for the single-player game with payoff function $g_i$.
A standard condition that guarantees the assumed supermartingality/martingality
properties is the integrability condition
\[\E_x\left[\sup_{t\geq 0}e^{-rt}g_i(X_t)\right]<\infty,\]
see, e.g., \cite{S}.
\end{rem}

The two players will be equipped with randomised stopping times, which we now define. 
Denote $\mathcal A$ the set of non-decreasing right-continuous $\mathbb{F}-$adapted processes $\Gamma=(\Gamma_t)_{0-\leq t<\infty}$ with values in $[0,1]$ and with $\Gamma_{0-}=0$.

\begin{definition} (Randomised stopping times.)
Let $U\sim \text{Unif}(0,1)$ be a random variable independent of $X$ (and also independent of the random variables $\theta_1$ and $\theta_2$ introduced below). Given $\Gamma\in\mathcal A$, 
the $(U,\Gamma)-$randomised stopping time $\gamma$ is defined as
	\[\gamma: = \inf\{t\geq0: \Gamma_t > U\}.\]
\end{definition}

Throughout the article we always assume with no further mentioning that two randomised stopping times $\gamma_1$ and $\gamma_2$ have independent randomisation devices $U_1$ and $U_2$.
Also, since the distribution of $\gamma$ only depends on $\Gamma$, we will also refer to the $(U,\Gamma)-$randomised stopping time as a $\Gamma-$randomised stopping time.

We will study a game where Player~$i$ has competition with probability $p_i\in(0,1]$, $i=1,2$, with $p_1\wedge p_2<1$.
To model this, let $\theta_1, \theta_2$ be Bernoulli random variables such that $\theta_1, \theta_2$ and $X$ are independent, and with
	\[\mathbb P_x(\theta_i=1)=p_i ,\]
$i=1,2$. We will let $\{\theta_i=1\}$ represent the event on which Player~$i$ has active competition.
To do that, if Player~$3-i$ uses a randomized stopping time $\gamma_{3-i}$, we define
\[\hat\gamma_{3-i}:=\left\{\begin{array}{ll}
\gamma_{3-i} & \mbox{on }\{\theta_{i}=1\}\\
\infty & \mbox{on }\{\theta_{i}=0\}.\end{array}\right.\]
The expected discounted payoff for Player $1$ is then defined as
	\[J_1(x;\gamma_1,\gamma_2) := \E_x[e^{-r \gamma_1} g_1(X_{\gamma_1})\mathds{1}_{\{\gamma_1< \hat \gamma_2\}}+e^{-r\gamma_2}V^{h_1}(X_{\gamma_2} )\mathds{1}_{\{\gamma_1\geq \hat\gamma_2 \}}],\] 
and the expected discounted payoff for Player 2 is defined as
	\[J_2(x;\gamma_1,\gamma_2) := \E_x[e^{-r \gamma_2} g_2(X_{\gamma_2})\mathds{1}_{\{\gamma_2\leq \hat\gamma_1\}}+e^{-r\gamma_1}V^{h_2}(X_{\gamma_1} )\mathds{1}_{\{\hat\gamma_1<\gamma_2 \}}].\]  
Note that in case of simultaneous stopping, Player 2 has priority and receives $g_2$, whereas Player~1 receives the consolation $V^{h_1}$; we refer to Player~$2$ as the advantaged player.

\begin{rem}
In case Player~$i$ is forestalled by their opponent, their (immediate) consolation payoff is specified by $V^{h_i}$. This corresponds to a situation where one, after being forestalled, plays a single-player stopping game with payoff function $h_i$.
\end{rem}

Given a control $\Gamma\in\mathcal A$, we denote by 
\begin{equation}
\label{gammai}
\gamma(u):=\inf\{t\geq 0:\Gamma_t>u\}
\end{equation}
the first time that $\Gamma$ exceeds $u\in[0,1)$; note that $\gamma(u)$ is a stopping time for each $u\in[0,1)$, and that
$\gamma=\gamma(U)$ where $U$ is the randomisation device. The following lemma is an immediate consequence of the law of total expectation.

\begin{lem}\label{lem}
Let $\Gamma^1,\Gamma^2\in\A$, denote by $\gamma_1$ and
$\gamma_2$ the corresponding randomised stopping times, and let $\gamma_i(u)$, $i=1,2$ be defined as in \eqref{gammai}. Then 
\begin{equation}\label{random}
J_1(x;\gamma_1,\gamma_2)= \int_0^1 J_1(x;\gamma_1(u),\gamma_2)\,du
\end{equation}
and 
\begin{equation}\label{random2}
J_2(x;\gamma_1,\gamma_2)=\int_0^1 J_2(x;\gamma_1,\gamma_2(u))\,du.
\end{equation}
\end{lem}

In the following result we provide simplified expressions for 
the integrands in \eqref{random}-\eqref{random2}. 

\begin{proposition}\label{prop}
If $\gamma_2$ is a $\Gamma^2-$randomised stopping time and $\tau$ is a stopping time, then
\begin{equation}\label{J1}
    J_1(x;\tau,\gamma_2)= \E_x\left[e^{-r\tau}(1-p_1\Gamma^2_{\tau})g_1(X_\tau)+p_1\int_{[0,\tau]} e^{-rt}V^{h_1}(X_t)\,d\Gamma^2_t\right].
    \end{equation}
Similarly, if $\gamma_1$ is a $\Gamma^1-$randomised stopping time, then 
\begin{equation}\label{J2}
    J_2(x;\gamma_1,\tau)=\E_x\left[e^{-r\tau}(1-p_2\Gamma^1_{\tau-})g_2(X_\tau)+p_2\int_{[0,\tau)} e^{-rt}V^{h_2}(X_t)\,d\Gamma^1_t\right].
    \end{equation}
\end{proposition}

\begin{proof}
First note that 
\[\{\Gamma^2_\tau<U_2\}\subseteq \{\tau<\gamma_2\}\subseteq \{\Gamma^2_\tau\leq U_2\}\]
so that $\mathbb{P}_x(\tau<\gamma_2\vert \F_{\tau})=1-\Gamma^2_\tau$.
Therefore we have 
\begin{eqnarray*}
\E_x\left[e^{-r \tau} g_1(X_{\tau})\mathds{1}_{\{\tau< \hat \gamma_2\}}\right] &=& \E_x\left[e^{-r \tau} g_1(X_{\tau})(\mathds{1}_{\{\theta_1=0\}}+
\mathds{1}_{\{\theta_1=1\}}\mathds{1}_{\{\tau<  \gamma_2\}})\right]\\
&=&(1-p_1)\E_x\left[e^{-r \tau} g_1(X_{\tau})\right] 
    +p_1\E_x\left[e^{-r \tau} g_1(X_{\tau})\mathds{1}_{\{\tau< \gamma_2\}}\right]\\
&=&(1-p_1)\E_x\left[e^{-r \tau} g_1(X_{\tau})\right] 
    +p_1\E_x\left[e^{-r \tau} g_1(X_{\tau})\mathbb{P}_x(\tau<\gamma_2\vert \F_{\tau})\right]\\
% &=& \E_x\left[e^{-r \tau} g_1(X_{\tau})\E_x\left[\mathds{1}_{\{\tau< \hat \gamma_2\}}\vert\F_\tau\right]\right]\\
&=& \E_x\left[e^{-r \tau} (1-p_1\Gamma^2_\tau)g_1(X_{\tau})\right],
\end{eqnarray*} 
where the second equality uses independence of $\theta_1$ and $X$.
Moreover, 
\begin{eqnarray*}
\E_x\left[e^{-r\gamma_2}V^{h_1}(X_{\gamma_2} )\mathds{1}_{\{\tau\geq \hat\gamma_2 \}}\right] 
&=& p_1\E_x\left[e^{-r\gamma_2}V^{h_1}(X_{\gamma_2} )\mathds{1}_{\{\tau\geq\gamma_2 \}}\right]\\
&=& p_1\E_x\left[\int_0^1 e^{-r\gamma_2(u)}V^{h_1}(X_{\gamma_2(u)} )\mathds{1}_{\{\tau\geq\gamma_2(u) \}}\,du\right]\\
&=& p_1\E_x\left[\int_{[0,\tau]} e^{-rt}V^{h_1}(X_{t} )\,d\Gamma^2_t\right],
\end{eqnarray*}
where the last equality follows from \cite[Proposition 4.9, p.~8]{RY} since $\gamma_2$ is the inverse of $\Gamma^2$. Thus \eqref{J1} follows.

    The argument for \eqref{J2} is similar, and we omit the details.
\end{proof}

It is clear from Lemma~\ref{lem} and Proposition~\ref{prop} that the functionals $J_1$ and $J_2$ only depend on the chosen controls $\Gamma^1$
and $\Gamma^2$ (and not on the randomisation devices). We may therefore identify a randomised stopping time $\gamma$ with its increasing control $\Gamma$; for example, we sometimes write $J_i(x;\Gamma^1,\Gamma^2)$ or $J_i(x;\Gamma^1,\gamma_2)$ instead of $J_i(x;\gamma_1,\gamma_2)$.

\begin{definition}
A pair $(\Gamma_1^*,\Gamma_2^*)\in\A^2$ is a Nash equilibrium if for any pair 
$(\Gamma_1,\Gamma_2)\in\A^2$ we have 
	\[J_1(x;\Gamma_1,\Gamma_2^*)\leq J_1(x;\Gamma_1^*,\Gamma_2^*) \text{ and }J_2(x;\Gamma_1^*,\Gamma_2)\leq J_2(x;\Gamma_1^*,\Gamma_2^*). \]
\end{definition}

 Our goal is to obtain conditions under which a Nash equilibrium $(\Gamma_1^*,\Gamma_2^*)$ exists, and to study its associated equilibrium values
 	\[u_1(x,p_1,p_2)=J_1(x;\Gamma_1^*,\Gamma_2^*) \quad\&\quad u_2(x,p_1,p_2)=J_2(x;\Gamma_1^*,\Gamma_2^*).\]

\section{A verification theorem}
In this section we provide a verification result, i.e. we specify conditions under which a Nash equilibrium can be constructed from two given functions 
$u_1,u_2:\I\times[0,1]^2\to[0,\infty)$. Here the functions $u_1(x,p_1,p_2)$ and $u_2(x,p_1,p_2)$ will play the role of equilibrium values for the obtained Nash equilibrium.

Given two control processes $\Gamma^1,\Gamma^2$, define the corresponding {\em adjusted belief process} 
\begin{equation}
    \label{Pi}
\Pi^i_t=\left\{\begin{array}{cl}
\frac{p_i(1-\Gamma^{3-i}_t)}{1-p_i\Gamma^{3-i}_t} & \mbox{if }p_i<1\\
1 & \mbox{if }p_i=1\end{array}\right.
\end{equation}
for $t\geq 0-$. Note that
\[\Pi^i_t=\mathbb P_x(\theta_i=1\vert \F_t,\hat\gamma_{3-i}>t),\]
so $\Pi^i$ is the conditional probability of active competition  for Player~$i$, conditional on observations of $X$ and on the event that $\hat\gamma_{3-i}$ has not yet occurred.

\begin{theorem}
		\label{verif}
Let two continuous functions $u_1,u_2:\I\times[0,1]^2\to[0,\infty)$ and a pair $(\Gamma^1,\Gamma^2)\in\A^2$ be given. 
Assume that $u_i\leq V^{g_i}$, and that $\Gamma^i_{\tau_{g_i}}=1$, $i=1,2$. Define on $[0,\infty)$ two processes
		\[M^1_t:=
			e^{-rt}(1-p_{1}\Gamma^{2}_{t})u_1(X_{t},\Pi^{1}_{t}, \Pi^2_{t}) + p_{1}\int_{[0,t]} e^{-rs}V^{h_1}(X_s)\,d\Gamma^{2}_s\]
   and 
   	\[M^2_t:=
			e^{-rt}(1-p_{2}\Gamma^{1}_{t-})u_2(X_{t},\Pi^{1}_{t-}, \Pi^2_{t-}) + p_{2}\int_{[0,t)} e^{-rs}V^{h_2}(X_s)\,d\Gamma^{1}_s,\]
and assume that 
\begin{itemize}
\item[(i)]
$M^i$ is a supermartingale, and it is a martingale 
on $[0,\gamma_i(u)]$ for any $u<1$, $i=1,2$;
\item[(i')] 
$M^2$ is continuous on $[0,\tau_{g_2}]$ and on $(\tau_{g_2},\infty)$, with 
$M^2_{\tau_{g_2}+}-M^2_{\tau_{g_2}}\leq 0$;
\item[(ii)]
$u_1(X_t,\Pi^1_{t},\Pi^2_{t})\geq g_1(X_t)$ and $u_2(X_t,\Pi^1_{t-},\Pi^2_{t-})\geq g_2(X_t)$ for all $t\geq 0$ $\mathbb P_x$-a.s.;
\item[(iii)]
$\Gamma^1_t=\int_{[0,t]} 1_{\{u_1(X_s,\Pi^1_{s},\Pi^2_{s})=g_1(X_s)\}}\,d\Gamma^1_s$ and $\Gamma^2_t=\int_{[0,t]} 1_{\{u_2(X_s,\Pi^1_{s-},\Pi^2_{s-})=g_2(X_s)\}}\,d\Gamma^2_s$.
\end{itemize}
Then $(\Gamma^1,\Gamma^2)$ is a Nash equilibrium, 
and the equilibrium values are given by 
\[M^1_0=(1-p_1\Gamma_0^2)u_1(x,\Pi^1_0,\Pi_0^2) + p_1\Gamma^2_0 V^{h_1}(x)\]
and $M^2_0=u_2(x,p_1,p_2)$, respectively.
	\end{theorem}
	
\begin{proof}
 Let $\tau$ be a stopping time. By optional sampling, 
 \begin{eqnarray*}
 M^1_{0} &\geq& \E_x[M_\tau^1]= \E_x\left[e^{-r\tau}(1-p_{1}\Gamma^{2}_{\tau})u_1(X_{\tau},\Pi^{1}_{\tau}, \Pi^2_{\tau}) + p_{1}\int_{[0,\tau]} e^{-rs}V^{h_1}(X_s)\,d\Gamma^{2}_s\right]\\
&\geq& \E_x\left[e^{-r\tau}(1-p_{1}\Gamma^{2}_{\tau})g_1(X_{\tau}) + p_{1}\int_{[0,\tau]} e^{-rs}V^{h_1}(X_s)\,d\Gamma^{2}_s\right]= J^1(x;\tau,\Gamma^2),
\end{eqnarray*}
where the last equality follows from Proposition~\ref{prop}.  
Moreover, by (i) and (iii), the inequalities reduce to equalities for the stopping times $\gamma_1(u)$, $u\in[0,1)$. Consequently, using Lemma~\ref{lem},
 \[J_1(x; \Gamma^1,\Gamma^2)=M_0^1=\sup_{\Gamma\in\A}J^1(x;\Gamma,\Gamma^2)\]
 so $\Gamma^1$ is an optimal response to $\Gamma^2$.

Similarly,
\begin{eqnarray*}
M_0^2 &\geq& \E_x[M_\tau^2] =
\E_x\left[e^{-r\tau}(1-p_{2}\Gamma^{1}_{\tau-})u_2(X_{\tau},\Pi^1_{\tau-},\Pi^2_{\tau-}) + p_{2}\int_{[0,\tau)} e^{-rs}V^{h_2}(X_s)\,d\Gamma^{1}_s\right]\\
&\geq& \E_x\left[e^{-r\tau}(1-p_{2}\Gamma^{1}_{\tau-})g_2(X_{\tau}) + p_{2}\int_{[0,\tau)} e^{-rs}V^{h_2}(X_s)\,d\Gamma^{1}_s\right]= J_2(x;\Gamma^1,\tau),
\end{eqnarray*}
with equality for $\tau=\gamma_1(u)$, $u\in[0,1)$.
Thus
\[J_2(x;\Gamma^1,\Gamma^2)=M_0^2=\sup_{\Gamma\in\A}J_2(x;\Gamma^1,\Gamma),\]
which completes the proof.
	\end{proof}

\section{Symmetric games with consolation}\label{sec4}

Obtaining explicit solutions to general problems exhibiting both asymmetry and consolation seems out of reach; however, some cases are amenable for further analysis. 
In this section we discuss how candidate equilibrium value functions 
$u_1$ and $u_2$ can be constructed for symmetric games.

To do that, assume that 
\begin{equation}
\label{symm}
g_1=g_2=:g\mbox{ and }h_1=h_2=:h
\end{equation} 
and that $p_1\leq p_2$.
Since both players share the same payoff functions, the player with the smallest probability of competition (Player~1) is less exposed to competition and should be more willing to stop late (i.e. at $\tau_g$). We thus expect that the equilibrium strategy $\Gamma^1$ of Player~1 should satisfy $\gamma_1(u)\to\tau_g$ as $u\to 1$, and in view of the indifference principle in game theory, 
the equilibrium value for Player~1 should be obtained by simply using $\gamma_1(1-)=\tau_g$. However, in that case one sees that 
\begin{equation}
\label{indiff}
J_1(x;\tau_g,\gamma_2)=(1-p_1)\E_x[e^{-r\tau_g}g(X_{\tau_g})]  +p_1\E_x[e^{-r\gamma_2}V^h(X_{\gamma_2})]\leq (1-p_1)V^g(x)+p_1V^h(x)
\end{equation}
for an equilibrium strategy $\gamma_2<\tau_g$, where the inequality uses supermartingality and optional sampling.

\subsection{Symmetric games with martingale consolation}
In addition to the symmetry condition \eqref{symm}, also assume that 
\begin{equation}
    \label{inclusion}
    e^{-rt\wedge\tau_g}V^h(X_{t\wedge\tau_g})\mbox{ is a martingale.}
    \end{equation}

\begin{rem}
In view of Assumption~\ref{ass}, a sufficient condition for \eqref{inclusion} to hold is that
\begin{equation}
\{x\in\I: V^g(x)>g(x)\}\subseteq\{x\in\I: V^h(x)>h(x)\},
\end{equation}
i.e. that the continuation region of $h$ contains the continuation region of $g$.\qed
\end{rem}

Under the assumption \eqref{inclusion}, the inequality in \eqref{indiff} is an equality, 
and the equilibrium value of Player~1 should be 
\begin{equation*}
u_1(x,p_1,p_2)=(1-p_1)V^g(x)+p_1V^h(x).
\end{equation*} 
If the process $\Gamma^2$ is constructed so that $\Pi^1_t$
reflects in the boundary $b$ defined by $(1-b(x))V^g(x)+b(x)V^h(x)=g(x)$,
then one needs by optimality considerations for Player~2 that also $u_2=g$ along the boundary. This leads to 
$u_2(x,p_1,p_2)=(1-p_1)V^g(x)+p_1V^h(x)$, and then $\Gamma^1$ can be chosen so that $M^2$ is a martingale. Note that the equilibrium values $u_1$ and $u_2$ then do not depend on $p_2$ (as long as $p_2\geq p_1$).

We now make the above heuristics precise.

\begin{theorem} %[Symmetric case with consolation]
\label{sym}
		Assume that $g_1=g_2=:g$, $h_1=h_2=:h$ and that \eqref{inclusion} holds.
Also assume that $p_1\leq p_2$.
Define a boundary 
\[b(x):=\left\{\begin{array}{cl}
0 & \mbox{if }V^g(x)=g(x)\\
\frac{V^g(x)-g(x)}{V^g(x)-V^h(x)} & \mbox{if }V^h(x)\leq g(x)<V^g(x)\\
1 & \mbox{if }g(x)<V^h(x),\end{array}\right.\]
  and let 
   \[\Gamma^2_t:=\frac{p_1-p_1\wedge \inf_{0\leq s\leq t} b(X_s)}{p_1(1-p_1\wedge \inf_{0\leq s\leq t} b(X_s))}\]
  and 
  \[\Gamma^1_t:=\left\{\begin{array}{cl}
  \frac{p_1(1-b(x)\wedge p_1)}{p_2(1-p_1)}(\Gamma^2_t-\Gamma^2_0) & t<\tau_g\\
  1 & t\geq \tau_g.\end{array}\right.\]
  Then $(\Gamma^1,\Gamma^2)$ is a Nash equilibrium, with corresponding equilibrium values
  \begin{equation}
\label{cand}
u_1(x,p_1):=(1-p_1)V^g(x)+p_1V^h(x)
\end{equation}
  and
  \[u_2(x,p_1):=\max\left\{(1-p_1)V^g(x)+p_1V^h(x), g(x)\right\}.\]
\end{theorem}

\begin{proof}
First note that it follows from the continuity conditions in Assumption~\ref{ass} that $b$ is lower semi-continuous; moreover, if
$\liminf_{y\to x}b(y)>b(x)$ for some $x\in\I$, then $V^g(x)=g(x)$.
Consequently, $t\mapsto \Gamma^2_t$ is continuous on $[0,\tau_g)$ and $t\mapsto \Gamma^1_t$ is continuous on $[0-,\tau_g)$, with $\Gamma^1_{\tau_g}=\Gamma^2_{\tau_g}=1$.
Furthermore,
\[(1-p_1\Gamma^2_t)(1-\Pi^1_t)=1-p_1\]
and 
\[(1-p_1\Gamma^2_t)\Pi^1_t = p_1(1-\Gamma^2_t)\]
for $t\geq 0-$, and 
\[(1-p_2\Gamma^1_{t-})(1-\Pi^1_{t-})=\left\{\begin{array}{cl}
1-p_1 & t=0\\
1- b(x)\wedge p_1 & 0<t\leq \tau_g\\
1-p_2 & t>\tau_g\end{array}\right.\]
and 
\[(1-p_2\Gamma^1_{t-})\Pi^1_{t-}=\left\{\begin{array}{cl}
p_1 & t=0\\
p_1\wedge b(x)-p_2\Gamma^1_{t-} & 0<t\leq \tau_g\\
0 & t>\tau_g.\end{array}\right.\]
Consequently,
\begin{eqnarray*}
M^1_t &:=& e^{-rt} (1-p_1\Gamma^{2}_t) u_1(X_{t},\Pi^1_{t})+ p_1\int_{[0,t]} e^{-rs} V^h(X_s)\,d\Gamma^{2}_s\\
&=& e^{-rt}(1-p_1)V^g(X_t)+e^{-rt}p_1(1-\Gamma^2_t)V^h(X_t)+p_1\int_{[0,t]}e^{-rs}V^h(X_s)\,d\Gamma^2_s,
\end{eqnarray*}
which is a supermartingale on $[0,\infty)$ and a martingale on $[0, \tau_g]$. Similarly, 
\begin{eqnarray*}
M^2_t &:=& e^{-rt} (1-p_2\Gamma^{1}_{t-}) u_2(X_{t},\Pi^1_{t-})+ p_2\int_{[0,t)} e^{-rs} V^h(X_s)\,d\Gamma^{1}_s\\
&=& \left\{\begin{array}{ll}
e^{-rt}(1-p_1\wedge b(x))V^g(X_t)+e^{-rt}(p_1\wedge b(x)-p_2\Gamma^1_{t-})V^h(X_t)\\ \hspace{45mm}+p_2\int_{[0,t)}e^{-rs}V^h(X_s)\,d\Gamma^1_s& t\leq \tau_g\\
e^{-rt}(1-p_2)V^g(X_t)+p_2\int_{[0,\tau_g]}e^{-rs}V^h(X_s)\,d\Gamma^1_s  & t> \tau_g.
\end{array}\right.
\end{eqnarray*}
At $t=\tau_g$, the process $M^2$ makes a jump of size
\[M^2_{\tau_g+}-M^2_{\tau_g}=-e^{-r\tau_g}(p_2-p_1\wedge b(x))(V^g(X_{\tau_g})-V^h(X_{\tau_g}))\leq 0,\]
and it is thus clear that $M^2$ is a supermartingale on $[0,\infty)$ and a martingale on $[0,\tau_g]$.
By the verification result (Theorem~\ref{verif}), the pair $(\Gamma^1,\Gamma^2)$
constitutes a Nash equilibrium, and the corresponding equilibrium values are given by $M_0^1=u _1(x,p_1)$ and $M^2_0=u_2(x,p_1)$, respectively.
\end{proof}

\begin{rem}\label{rem2}
In the main result of \cite{DE}, an equilibrium was constructed for the symmetric case with $g:=g_1=g_2$ and $h_1=h_2=0$.
Since condition \eqref{inclusion} is trivially satisfied in that case, 
Theorem~\ref{sym} can be viewed as the extension of the main result in \cite{DE} to a set-up allowing for (martingale) consolation; note, however, that in case of simultaneous stopping the payoff in \cite{DE} was specified to be evenly distributed between players and thus slightly different from the present set-up. 
\end{rem}

\begin{example}\label{ex1}
Let $X$ be a geometric Brownian motion
     \[dX_t = \mu X_t \,dt+ \sigma X_t\, dW_t, \quad X_0 = x, \]
     where $\sigma>0$ and $\mu<r$ are constants and $W$ is a standard Brownian motion, and consider the game specification $g(x)=g_1(x)=g_2(x)=(x-K)^+$ and 
$h(x)=h_1(x)=h_2(x)=(x-L)^+$, for positive constants $K<L$. 
This can be seen as a model for investment in a real option under possible competition, where the investment cost is larger for the second investor (for the non-competitive case, see \cite{dixit} and \cite{MS}). 
Then 
\[V^g(x)=\left\{\begin{array}{ll}
(b_g-K)(x/b_g)^\gamma & x<b_g\\
x-K & x\geq b_g,\end{array}\right.\]
where 
\[b_g=\frac{\gamma K}{\gamma-1}\]
and $\gamma>1$ is the unique positive solution of 
\begin{equation}
\label{gamma}
\frac{\sigma^2}{2}\gamma(\gamma-1)+\mu \gamma-r=0,
\end{equation}
with similar expressions for $V^h$ and $b_h$.

By Theorem~\ref{sym}, an equilibrium is obtained if 
 \[\Gamma^1_t:=\left\{\begin{array}{cl}
  \frac{p_1(1-b(x)\wedge p_1)}{p_2(1-p_1)}(\Gamma^2_t-\Gamma^2_0) & t<\tau_g\\
  1 & t\geq \tau_g,\end{array}\right.\]
and
$\Gamma^2$ is chosen so that $(\Pi^1,X)$ reflects in the boundary 
\[b(x):=\left\{\begin{array}{cl}
0 & \mbox{if }x\geq b_g\\
\frac{V^g(x)-g(x)}{V^g(x)-V^h(x)} & \mbox{if }V^g(x)>g(x)\geq V^h(x)\\
1 & \mbox{if }x<a,\end{array}\right.\]
where $a\in[K,b_g]$ is the unique positive solution of $V^h(a)=g(a)$. 
Moreover, the corresponding equilibrium values are given by $u_1(x,p_1)=(1-p_1)V^g(x)+p_1V^h(x)$ and $u_1(x,p_1)=\max\{(1-p_1)V^g(x)+p_1V^h(x), g(x)\}$, respectively.
For a graphical illustration of
the reflected process, see Figure~\ref{fig1}.

\begin{figure}[htp!]
%\begin{figure}[h]
    \centering
    \includegraphics[width=0.7\textwidth]{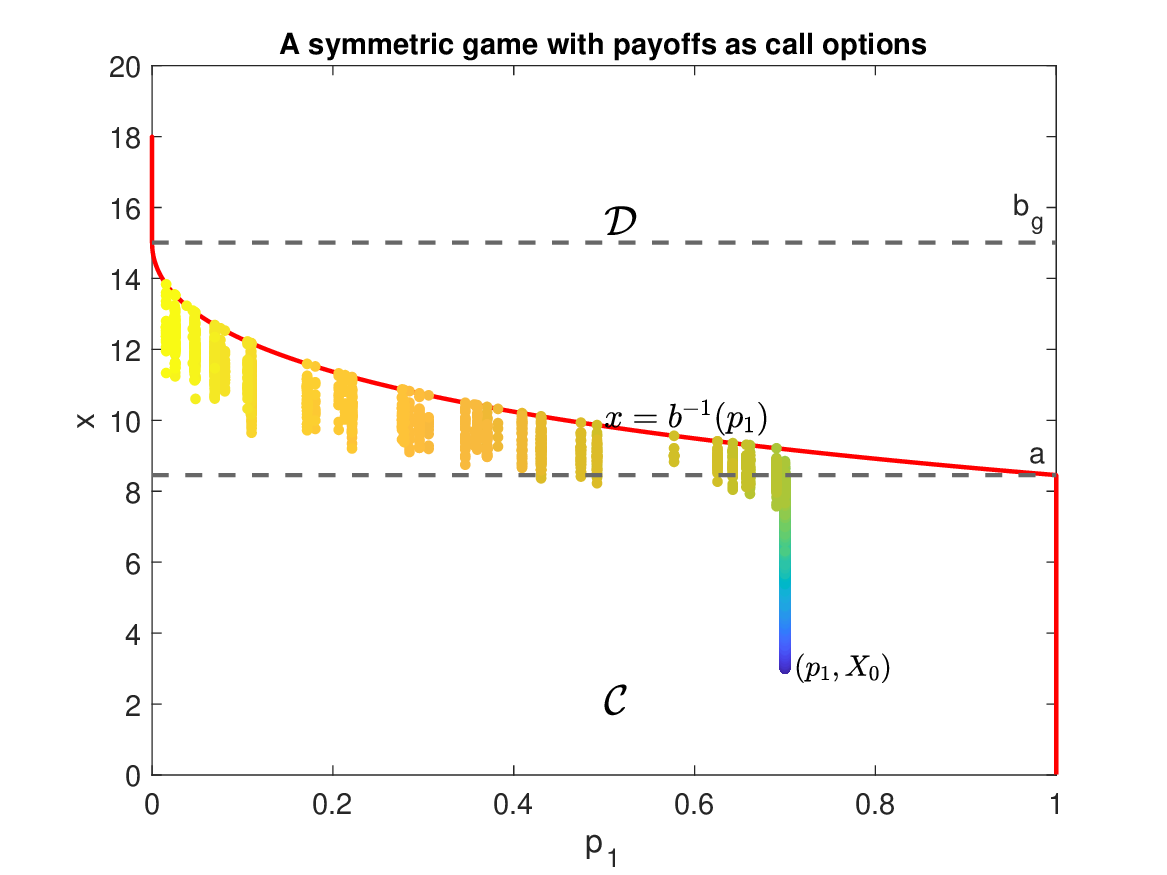}
    \caption{The  boundary $b$ and a simulated path of the reflected process $(\Pi^1,X)$ from
Example~\ref{ex1}. Here $g(x) = (x-3)^+$, $h(x) = (x-4)^+$, $\mu = 0.08$, $\sigma = 0.01$ and $r = 0.1$.}
    \label{fig1}
\end{figure}
\end{example}

\subsection{Symmetric games with supermartingale consolation}\label{sec6}

We now consider the symmetric case with $g:=g_1=g_2$ and $h:=h_1=h_2$, but without the assumption \eqref{inclusion} of martingale consolation. 
Again we impose the condition $p_1\leq p_2$, and note that the same argument as in \eqref{indiff} gives a candidate value 
\begin{equation}
    \label{u1}
u_1(x,p_1,p_2)=(1-p_1)V^g(x) + p_1\E_x [e^{-r\gamma_2}V^h(X_{\gamma_2})]
\end{equation}
for Player~1, 
where $\gamma_2$ is the equilibrium strategy of Player~2. Due to supermartingality, however, the second term is no longer explicitly known, so there is no immediate way to construct candidate equilibrium values. 

On the other hand, it may still be argued that the equilibrium value function $u_1$ is independent of $p_2$, provided $p_1\leq p_2$. An informal motivation for this is as follows. First note that  $u_1(x,p_1,p_2)$ should be non-increasing in $p_2$ since a larger belief in competition suggests that Player~2 should stop early, which decreases the value for Player~1. At the same time, the first term on the right-hand side in \eqref{u1} does not depend on $p_2$, whereas the second term exhibits the opposite monotonicity: a large $p_2$ would suggest a small $\gamma_2$, and supermartingality then suggests that $u_1$ is non-decreasing in $p_2$. 
Consequently, we thus expect that $u_1=u_1(x,p_1)$ provided $p_1\leq p_2$.
Then, for 
\[M^1_t=e^{-rt}(1-p_{1}\Gamma^{2}_{t})u_1(X_{t},\Pi^{1}_{t}) + p_{1}\int_{[0,t]} e^{-rs}V^{h_1}(X_s)\,d\Gamma^{2}_s\]
to be a martingale, one needs the condition 
\begin{equation}
\label{bc}
(1-p_1)\frac{\partial u_1}{\partial p_1} + u_1=V^h
\end{equation}
to hold at points of increase of the control $\Gamma^2$.

\begin{example} 
Consider the symmetric case with $g(x)=g_1(x)=g_2(x)=(x-K)^+$ where $K>0$ is a given constant,
$h=h_1=h_2\leq g$, and assume that $p_1\leq p_2$.
As in Example~\ref{ex1} above, let 
 \[b_g:=\frac{\gamma K}{\gamma-1}\]
be the single-player optimal boundary. For simplicity, we also assume that $h(x)<g(x)$ for $x>K$, and  denote 
$a$ the unique solution to $V^h(a)=g(a)$ in $[K,b_g]$. 

We make the Ansatz 
\[u_1(x,p_1)=c(p_1)\psi(x)\]
for $x\leq b^{-1}(p)$, where $\psi(x):=x^\gamma$, and $b$ and $c$ are yet to be determined. 
From the condition $u_1=g$ at boundary points $(x,b(x))$ and the martingale condition
\eqref{bc}, we get the system
\[\left\{\begin{array}{ll}
   c(b(x))\psi(x)=g(x)    \\
   (1-b(x))c'(b(x))\psi(x) + c(b(x)) \psi(x)= V^h(x).
\end{array}\right.\]
Eliminating $c$, this leads to the ODE 
\begin{equation}
    \label{bODE}
(1-b(x))(\frac{g}{\psi})'(x) \psi(x)+(g(x) -V^h(x))b'(x)=0,
\end{equation}
with solution
\begin{equation}\label{der b}
b(x)=1-\exp\left\{-\int^{b_g}_x\frac{(\frac{g(y)}{\psi(y)})'\psi(y)}{g(y)-V^h(y)}dy\right\},
\end{equation}
where we also imposed the boundary condition $b(b_g)=0$. 

For the verification, we thus define 
 \[b(x) := \left\{\begin{array}{cl}
 0 & x\geq b_g\\
 1-\exp\left\{-\int^{b_g}_{x} \frac{1-\gamma+\frac{K\gamma}{y}}{y-K-V^h(y)}dy\right\} & x\in(a,b_g)\\
 1 & x\leq a,\end{array}\right.\]
and note that $b$ is continuous.
Let 
\begin{equation}\label{G2}
\Gamma^2_t:=\frac{p_1-p_1\wedge \inf_{0\leq s\leq t} b(X_s)}{p_1(1-p_1\wedge \inf_{0\leq s\leq t} b(X_s))}
\end{equation}
  and 
  \begin{equation}\label{G1}
  \Gamma^1_t:=\left\{\begin{array}{cl}
  \frac{p_1}{p_2}(\Gamma^2_t-\Gamma^2_0) & t<\tau_g\\
  1 & t\geq \tau_g.\end{array}\right.
 \end{equation}
Also, let 
\begin{equation}
    u_1(x,p_1):=\left\{\begin{array}{cl}
V^g(x) & p_1=0\\
\frac{g(b^{-1}(p_1))}{\psi(b^{-1}(p_1))}\psi(x) & 0<p_1\leq b(x) \\
\frac{1-p_1}{1-b(x)}g(x) + \frac{p_1-b(x)}{1-b(x)} V^h(x) & p_1>b(x)\end{array}\right.
\end{equation}
and
\[u_2(x,p_1):=\left\{\begin{array}{cl}
V^g(x) & p_1=0\\
\frac{g(b^{-1}(p_1))}{\psi(b^{-1}(p_1))}\psi(x) &
0<p_1\leq b(x)\\
g(x) & p_1>b(x).\end{array}\right.\]

Then 
\[M^1_t := e^{-rt}(1-p_{1}\Gamma^{2}_{t})u_1(X_{t},\Pi^{1}_{t}) + p_{1}\int_{[0,t]} e^{-rs}V^{h}(X_s)\,d\Gamma^{2}_s\]
   and 
   	\[M^2_t:=
			e^{-rt}(1-p_{2}\Gamma^{1}_{t-})u_2(X_{t},\Pi^{1}_{t-}) + p_{2}\int_{[0,t)} e^{-rs}V^{h}(X_s)\,d\Gamma^{1}_s\]
fulfill the conditions in Theorem~\ref{verif}, so
$(\Gamma^1,\Gamma^2)$ is a Nash equilibrium with corresponding equilibrium values 
$u_1(x,p_1)$ and $u_2(x,p_1)$.
\end{example}

\section{Asymmetric games}
\label{sec7}

In the symmetric games studied above, the key insight that leads to their solutions is the intuition that Player~1 has an incentive to stop late (compared to Player~2) provided $p_1\leq p_2$.
This intuition extends to some asymmetric games.
We illustrate this in the following example which accommodates both asymmetry and consolation.

\begin{example}\label{ex5.1}
    Consider a situation where each investor has an individual
    investment cost, which also goes up if forestalled.
   More precisely, let $g_i (x) = (x-K_i)^+$ and $h_i(x) = (x-L_i)^+$, with $K_i<L_i$ for $i = 1,2$, and with $K_2\leq K_1$. As usual, the value functions in the corresponding single-player games are given by  
     \[V^{g_i}(x)=\left\{\begin{array}{cl}
     (b_{g_i}-K_i)(\frac{x}{b_{g_i}})^{\gamma} & x<b_{g_i}\\
     x-K_i & x\geq b_{g_i},\end{array}\right.\]
     where 
    \[b_{g_i}=\frac{\gamma K_i}{\gamma-1}.\]
Denote $a_{i}$ the point where $V^{h_i}(a_i) = g_i(a_i)$, $i=1,2$. 
% In other words, $a_1 \in (K_1, b_{g_1})$ satisfies 
% \[\frac{a_1-K_1}{a_1^{\gamma}} = \left(\frac{\gamma-1}{\gamma L_1}\right)^{\gamma-1}\frac{1}{\gamma}.\]
Furthermore, we assume that 
\[a_2< a_1<b_{g_2}\]
and $p_1 \leq p_2$.
     Since we also have $b_{g_2}\leq b_{g_1}$, Player~1 is then naturally inclined to 
wait longer than Player~2. Therefore, a natural Ansatz for the equilibrium value of 
Player $1$ is 
         \[u_1(x,p_1)= (1-p_1)V^{g_1}(x)+p_1 V^{h_1}(x).\]
Define a boundary function 
\[b(x):=\left\{\begin{array}{cl}
0 & x\geq b_{g_1}\\
\frac{V^{g_1}(x)-g_1(x)}{V^{g_1}(x)-V^{h_1}(x)} & a_1<x<b_{g_1}\\
1 & x \leq a_1.\end{array}\right.\]
Then $u_1(x,p_1)\geq g_1(x)$ for all $x$ with $p_1\leq b(x)$. Define a process
    \[\Gamma^2_t:=\left\{\begin{array}{cl}
\frac{p_1-p_1\wedge \inf_{0\leq s\leq t} b(X_s)}{p_1(1-p_1\wedge \inf_{0\leq s\leq t} b(X_s))} 
& t\leq \tau_{g_2}\\
\Gamma^2_{\tau_{g_2}} & \tau_{g_2}<t<\tau_{g_1}\\
1 & t\geq \tau_{g_1}\end{array}\right.\]
as a candidate strategy for Player $2$, and let 
\[u_2(x,p_1):=\left\{\begin{array}{cl}
V^{g_2}(x) & p_1=0\\
\frac{g_2(b^{-1}(p_1))}{\psi(b^{-1}(p_1))}\psi(x) &
0<p_1\leq b(x)\\
g_2(x) & p_1>b(x).\end{array}\right.\]
We have
\begin{eqnarray*}
M^1_t &:=& e^{-rt} (1-p_1\Gamma^{2}_t) u_1(X_{t},\Pi^1_{t})+ p_1\int_{[0,t]} e^{-rs} V^{h_1}(X_s)\,d\Gamma^{2}_s\\
&=& e^{-rt}(1-p_1)V^{g_1}(X_t)+e^{-rt}p_1(1-\Gamma^2_t)V^{h_1}(X_t)+p_1\int_{[0,t]}e^{-rs}V^{h_1}(X_s)\,d\Gamma^2_s,
\end{eqnarray*}
which is a martingale on $[0, \tau_{g_1}]$ and a supermartingale on $[0,\infty)$. Similarly, for an arbitrary $\Gamma^1$ we have
\begin{eqnarray*}
M^2_t &:=& e^{-rt} (1-p_2\Gamma^{1}_{t-}) u_2(X_{t},\Pi^1_{t-})+ p_2\int_{[0,t)} e^{-rs} V^h(X_s)\,d\Gamma^{1}_s,
\end{eqnarray*}
so if we define $\Gamma_t^1$ by letting $\Gamma_0^1=0$, $\Gamma_{\tau_{g_2}}=1$ and
\begin{eqnarray*}
    p_2(V^{h_2}(X_t)-u_2(X_t,\Pi_t^1))d\Gamma_t^1- p_1(1-p_1)\frac{\partial u_2}{\partial p_1} (X_t,\Pi_t^1)\frac{1-p_2\Gamma_t^1}{(1-p_1\Gamma_t^2)^2}d\Gamma_t^2  = 0
\end{eqnarray*}
on $(0,\tau_{g_2})$, then $M^2$ is a martingale on $[0,\tau_{g_2}]$ and a supermartingale on $[0,\infty)$. By the verification result, the pair $(\Gamma^1,\Gamma^2)$ is a Nash equilibrium,
and the corresponding equilibrium values are given by $u_1$ and $u_2$.
\end{example}

The above solution method breaks down for most asymmetric games. 
We illustrate this in the following example.

%\subsection{Type II stopping surface: only Player 1 stops with some intensity}
%        Assume that the surface can be written as 
%    \[\{(x,p_1,p_2): x = b_2(p_1,p_2)\}.\]
%On such a surface, 
%\[u_1(x,p_1,p_2) = g_1(x).\]
%    Furthermore, $d\Gamma_t^1>0, d\Gamma_t^2=0$. As a consequence, $d\Pi_t^1<0, d\Pi_t^2=0$. Similarly, $dM_t^1 = 0$ implies that 
%    \[e^{-rt}\left(-p_2(1-p_2) \frac{1-p_1\Gamma_t^2}{(1-p_2\Gamma_t^1)^2} \frac{\partial u_1}{\partial p_2}(X_t, \Pi_t^1,\Pi_t^2) \right)d\Gamma_t^1 = 0,\]
%     and $dM_t^2 = 0$ implies that 
%     \[u_2(X_t, \Pi_t^1,\Pi_t^2) + (1-p_2)\frac{\partial u_2}{\partial p_2}(X_t, \Pi_t^1,\Pi_t^2) -V^{h_2}(X_t) = 0.\]
%     Plugging in $\Gamma_0^1 = \Gamma_0^2 = 0$, we have 
%     \begin{align*}
%         \begin{cases}
%             \frac{\partial u_1}{\partial p_2}(b_2, p_1,p_2) = 0,\\
%             u_2(b_2, p_1,p_2)+(1-p_2)\frac{\partial u_2}{\partial p_2}(b_2, p_1,p_2)-V^{h_2}(b_2) = 0.
%         \end{cases}
%     \end{align*}
%     We have three boundary conditions to solve for $u_1,u_2,b_2$.

\begin{example}\label{ex5.2}
\label{put-call}
Consider the case where Player $1$ and Player $2$ have a put and a call option as their payoffs, respectively, and with no consolation, i.e. $g_1(x) = (K_1-x)^+$, $g_2(x) = (x-K_2)^+$ and $h_1=h_2 = 0$, with $K_1<K_2$. 
%When $K_1\leq x\leq K_2$, both players would continue observing and the adjusted belief processes remain constants $p_1, p_2$. 
Then, clearly, neither player has an incentive to wait longer than their opponent; instead, if 
$X$ goes sufficiently below $K_1$ then Player~1 would stop, and if $X$ goes sufficiently above $K_2$ then Player~2 would stop.

In view of this, one expects a lower boundary surface $\{x=L(p_1,p_2)\}$ and an upper boundary surface $\{x=U(p_1,p_2)\}$ such that
Player $1$ exercises their put option with some (generalized) intensity on $L$ so that $(X_t,\Pi^1_t, \Pi^2_t)$ reflects in $L$, while Player $2$ does nothing, and Player $2$ exercises their call option with some intensity on $U$ so as to reflect $(X_t,\Pi^1_t, \Pi^2_t)$ in $U$ while Player $1$ does nothing.
Assuming that $X$ is a geometric Brownian motion as in previous examples, we then have candidate equilibrium values given by 
%
%Denote the American stopping boundaries of Player $i$ as $b_i$. There exists a lower boundary of Type II, $b_1<L(p_1, p_2)<K_1$, and an upper boundary  of Type I, $K_2<U(p_1, p_2)<b_2$ such that Player $1$ exercises their put option with some intensity on $\{(x,p_1, p_2):L(p_1, p_2)<x<K_1 \}$, while Player $2$ does nothing, and Player $2$ exercises their call option with some intensity on $\{(x,p_1, p_2):K_2<x<U(p_1, p_2) \}$, while Player $1$ does nothing. Assume 
\begin{align*}
\begin{cases}
u_1(x,p_1,p_2) = C_1 (p_1,p_2)\psi(x) + D_1 (p_1,p_2) \phi(x) ,\\
u_2(x,p_1,p_2) = C_2 (p_1,p_2) \psi(x) + D_2 (p_1,p_2) \phi(x) ,
\end{cases}
\end{align*}
where  $\psi = x^{\gamma}$ and $\phi=x^{\eta}$, with $\gamma>1$ and $\eta<0$ being the solutions of the quadratic equation \eqref{gamma}.
Clearly, boundary conditions $u_1=g_1$ and $u_2=g_2$ should be imposed on $L$ and on $U$, respectively.
Moreover, on 
$\{x = L(p_1, p_2)\}$, martingality of $M^1$ and $M^2$ requires that
\begin{align*}
\begin{cases}
     \frac{\partial u_1}{\partial p_2}= 0,\\
    (1-p_2)\frac{\partial u_2}{\partial p_2}+u_2 = 0.
\end{cases}
\end{align*}
Similarly, on 
$\{x=U(p_1,p_2)\}$, martingality of $M^2$ requires
\begin{align*}
\begin{cases}
     \frac{\partial u_2}{\partial p_1}= 0,\\
    (1-p_1)\frac{\partial u_1}{\partial p_1}+u_1 = 0.
\end{cases}
\end{align*}
Now, there are six boundary conditions imposed, and there are six unknown functions 
$C_1$, $D_1$, $C_2$, $D_2$, $L$ and $U$, and it is possible to derive a pair of candidate equilibrium values $u_1$ and $u_2$. In the interest of brevity, however, we refrain from 
doing so.
\end{example}

\begin{rem}
In Example~\ref{ex5.1}, a candidate solution is constructed (and verified) using the intuition that in equilibrium the players should increase their controls $\Gamma^1$ and $\Gamma^2$ simultaneously, with a relative weight specified so that martingality is obtained.
On the other hand, Example~\ref{ex5.2} describes a situation in which the players should {\em not} increase their controls simultaneously, and the boundary conditions for the candidate value functions given by martingality are thus specified accordingly. 

In general, one would hope to formulate a variational inequality which produces a unique pair of candidate value functions, with the correct boundary conditions along pieces of the boundary where both players increase their controls and along pieces where only one player does so. The exact formulation of such a variational problem, along with related existence and uniqueness questions, is currently not available.
\end{rem}

	\bibliographystyle{amsplain}
	\bibliography{ghost}

\end{document}